\title{\vspace{-1cm}Propagation of singularities by Osgood vector fields \\ and for 2D inviscid incompressible fluids}
\date{}
\documentclass[12pt]{article}
\usepackage{authblk}
\usepackage{blindtext,xcolor}
\usepackage{amsfonts}
\usepackage{amsmath}
\usepackage{amsthm,amssymb}
\usepackage{graphicx}
\usepackage{hyperref}
\usepackage[margin=2cm]{geometry}
\newtheorem{theorem}{Theorem}
\newtheorem{lemma}{Lemma}

\newtheorem{remark}{Remark}

\newcommand{\be}{\begin{equation}}
\newcommand{\ee}{\end{equation}}
\newcommand{\rmd}{{\rm d}}

\author{Theodore D. Drivas\footnote{Department of Mathematics, SUNY Stony Brook, Stony Brook, NY 11794 and School of Mathematics, Institute for Advanced Study, 1 Einstein Dr., Princeton, NJ 08540 ({\tt tdrivas@math.stonybrook.edu, tdrivas@ias.edu})}, \ Tarek M. Elgindi\footnote{Mathematics Department, Duke University, 120 Science Drive, Durham, NC 27708-0320  ({\tt tarek.elgindi@duke.edu})} \ 
and Joonhyun La\footnote{Department of Mathematics, Stanford University, Stanford, CA 94305 ({\tt joonhyun@stanford.edu})}}

\begin{document}

\maketitle
\vspace{-8mm}

\begin{abstract}
We show that certain singular structures (H\"{o}lderian cusps and mild divergences) are transported by the flow of homeomorphisms generated by an Osgood velocity field.  The structure of these singularities is related to the modulus of continuity of the velocity and the results are shown to be sharp in the sense that slightly more singular structures cannot generally be propagated.  For the 2D Euler equation, we prove that certain singular structures are preserved by the motion, e.g. a system of $\log\log_+(1/|x|)$ vortices (and those that are slightly less singular) travel with the fluid in a nonlinear fashion, up to bounded perturbations.  We also give stability results for weak Euler solutions away from their singular set.
\end{abstract}

The study of existence and uniqueness of solutions of ordinary differential equations and the corresponding transport PDE is an old subject.  By Peano's existence theorem, a bounded and continuous vector field $u(x,t):\mathbb{R}^d\times \mathbb{R}\to \mathbb{R}^d$ produces at least one solution of
\begin{align}\label{flowmap}
\frac{\rmd}{\rmd t} \phi(x,t) &= u(\phi(x,t), t), \qquad \phi(x,0) = x,
\end{align}
 for all $t\in \mathbb{R}$ and each $x\in \mathbb{R}^d$.   The classical Cauchy-Lipschitz theorem guarantees uniqueness of this solution provided that the velocity field is  Lipschitz continuous.  More generally, if we suppose that $u: \mathbb{R}^d\times \mathbb{R} \rightarrow \mathbb{R}^d$ has modulus of continuity $L: (0,\mathsf{m}_L) \rightarrow \mathbb{R}^+$ 
 with $\mathsf{m}_L>0$, that is
\begin{equation}\label{modulus}
|u(x, t) - u(y,t)| \le \|u(t)\|_L L(|x-y|),
\end{equation}
then the sharp conditions for uniqueness of solutions to Equation \eqref{flowmap} is that of Osgood (see e.g. \cite{H02,BCD11}).  The requirement on $L$ for it to be an Osgood modulus is that  $\lim_{z \rightarrow 0^+} \mathcal{M}(z) = \infty$ where 
\begin{equation}\label{mdef}
\mathcal{M}(z) := \int_z ^{\mathsf{m}_L} \frac{\rmd r}{L(r)}.
\end{equation}
We recall some examples (we will denote by $\log_+(z) = \max\{0, \log(z)\}$):  
\begin{itemize}
\item if $u$ is Lipschitz, i.e. $L(z) = z$ with $\mathsf{m}_L=1$, then $\mathcal{M}(z) = \log_+ (1/z)$,
\item  if $u$ is log-Lipschitz, i.e. $L(z) = z\log(1/z)$ with $\mathsf{m}_L=1/e$, then $\mathcal{M}(z) = \log \log_+ (1/z)$,
\item  generally, if $L(z) = z\log(1/z)\log_2(1/z)\dots \log_n(1/z) $   and $\mathsf{m}_L=1/e_{m}(1)$, with $n\in \mathbb{N}$ and 
$
\log_n (z)
=
\underbrace{\log\log\dots\log}_{\text{$n$ times}} z ,
$
then $\mathcal{M}(z) = \log_{1+n} (1/z)$.

\end{itemize}

A related issue is uniqueness of solutions to the corresponding transport equation
\begin{align} \label{transport}
\partial_t \theta + u \cdot \nabla \theta &= 0,\\
\theta|_{t=0}&=\theta_0.
\end{align}
It is known by the works \cite{AB08,CC21} that, if $u$ is an Osgood velocity field, then all integrable weak solutions are unique and admit a Lagrangian representation using the flowmap $\phi(x,t)$ \eqref{flowmap}
\be\label{lagrangianformula}
\theta(x,t) = \theta_0 (\phi^{-1}(x,t)). 
\ee
Of course, being able to uniquely solve the equations and write down the formula \eqref{lagrangianformula} does not imply that one knows how the solutions behave in any quantitative way. That is to say, while we may know that smooth initial data leads to smooth solutions, we may not actually know how the solutions look.  Moreover, the specific details would seem to depend very sensitively on the precise form of the advecting velocity.

In this paper, we show that certain (singular)  features which -- if present in the initial conditions $\theta_0$ of \eqref{transport} -- are preserved by all Osgood vector fields of a certain modulus (Theorem \ref{thmnormprop} and  \ref{thmnormprop2} herein).
As special case of our Theorem \ref{thmnormprop2} is: if $L$ is an Osgood modulus and $\mathcal{M}$ is given by \eqref{mdef}, then the solution $\theta(x,t)$ of the transport equation \eqref{transport}  supplied with initial datum
\begin{equation}\label{theta0dataint}
\theta_0 (x) = \mathcal{M} (|x - x_0 | ) + b_0(x)
\end{equation}
takes the following form for all $t\in \mathbb{R}$ 
\begin{equation}\label{thetasolutionint}
\theta(x,t) = \mathcal{M} (|x - \phi(x_0, t) | ) + b(x,t),
\end{equation} 
with remainder $b(x,t)$ which is bounded for all finite times. As can be seen from the examples above, the structures  $ \mathcal{M} (|x| ) $ are mildly singular in nature, possessing infinitely small scales while at the same time maintaining integrability.  Their shape and amplitude are unaltered by the motion induced by any (compatible in the above sense) Osgood field.  This fact is indicative of the inability for fields of that regularity to alter such fine structures due to not possessing sufficiently strong small scale motions themselves.  We also show by example that, for vector fields with a given Osgood modulus, any singularity steeper than $\mathcal{M}$ in the initial data will not in general be propagated in the solution of the transport equation.   As such, persistence of these singular structures provides insight into sharp borderline phenomena in linear transport.  In the nonlinear setting, some special structures such as {algebraic} singularities  \cite{Elling,BM} and vortex patches \cite{Chem,S94,BC,EJ} can be analyzed in detail. However, note that even analytic velocity fields do not in general allow for propagation of {algebraic} singularities -- their preservation seems to require additional structure \cite{Elling}.

We also discuss some application of these ideas to the two-dimensional Euler equations governing the motion of the velocity field $u$ of a perfect (inviscid and incompressible) fluid. These equations can be written as a nonlinear and nonlocal transport equation for the vorticity field $\omega = \nabla^\perp \cdot u$ where $\nabla^\perp = (-\partial_2, \partial_1)$:
\begin{align}\label{ee1}
\partial_t \omega + u \cdot \nabla \omega &=0, \\
u=\nabla^\perp &\Delta^{-1} \omega,\\
\omega|_{t=0}&=\omega_0.\label{ee2}
\end{align}
For the Euler equations, we establish a nonlinear analogue of our propagation of singularity result for linear transport.  Specifically, we show that if 
the  initial vorticity takes the form
\begin{equation}\label{om0dataint}
\omega_0 (x) = \log\log_+(1/|x | ) + b_0(x), \qquad b_0\in L^\infty(\mathbb{R}^2)\cap L^1(\mathbb{R}^2),
\end{equation}
then there exists a scalar function $b \in L_{loc}^\infty(\mathbb{R}; L^\infty(\mathbb{R}^2)\cap L^1(\mathbb{R}^2))$  and a $C^1$ map $\phi_*(t): \mathbb{R}\to \mathbb{R}^2$ such that  the unique weak solution $\omega(t)$ of 2D Euler emanating from $\omega_0$ takes the form:
\begin{equation}\label{omsolutionint}
\omega(x,t) =  \log\log_+(1/|x - \phi_*(t) |  ) + b(x,t).
\end{equation}
In fact, a more general statement is proved  for a remainder $b$ propagating in a suitable Yudovich space $Y_\Theta$.
Also, any finite number of such singular vortices can be evolved, each feeling each-others induced velocity field and moving accordingly. 
This is reminiscent of  the vortex-wave system where the interaction of multiple point-mass vortices with a perturbation is studied. However, due to their highly singular nature point-vortices do not actually solve the Euler equations even in a weak sense \cite{sch} and the studies which we are aware of actually require that the perturbations be constant in a neighborhood of the vortices. Our result gives a propagation of singular vortices for genuine 2D Euler solutions without any restriction on the perturbation aside from boundedness. 

We would like to mention here a prior result on propagation of singularities for 2D Euler due to Shnirelman \cite{S97}.  In that work, Euler solutions at finite (Besov) regularity are studied and it is shown that the most singular behavior (the ``germ" at infinity) of a certain vorticity-like quantity constructed from the Lagrangian flowmap is a generalized Lyapunov functional for Euler and is, in a sense, increasing on an open dense set in the phase space studied. From this, new integrals of motion are also constructed.  These results, while distinct from our own here, bare some qualitative similarities which would be interesting to explore in more detail.

All the above results holds in an existence/uniqueness class for the transport and Euler equation.  In \S \ref{DMclass}, we obtain a conditional uniqueness result for vorticity in $L^p$ with $p\geq 2$ (for which existence holds by DiPerna-Majda \cite{DM87}, but not uniqueness according to Vishik \cite{Vishik1,Vishik2}). Namely, the singular structure determines the whole solution in a precise sense.  As a consequence, if two solutions with the same initial data have isolated vortices with algebraic divergences propagating in exactly the same way in an analytic (or Gevrey) regular background, then the two solutions must coincide.

\section{Propagation of singularities for linear transport}\label{lintrans}

Here we study the propagation of singularities by the flow generated from any Osgood vector field, whose definition we recall here.
Suppose that $u: \mathbb{R}^d \rightarrow \mathbb{R}^d$ has modulus of continuity
\begin{equation}
L: (0,\mathsf{m}_L) \rightarrow \mathbb{R}^+
\end{equation}
where $\mathsf{m}_L>0$ is a suitable small number. 
We denoteby  $\phi$  the flow generated by $u$. It satisfies the ordinary differential equation \eqref{flowmap}.
Then, we have
\begin{equation}
|\phi(x,t) - \phi(y,t)| \le |x-y| + \int_0 ^t \| u(s) \|_{L} L(|\phi(x,s) - \phi(y,s) | ) \rmd s.
\end{equation}
Let $\mathcal{M}$ be given by \eqref{mdef}.
The vector field $u$ is called \emph{Osgood} continuous if $\lim_{z \rightarrow 0^+} \mathcal{M}(z) = \infty$. This, in turn, implies $\lim_{r\rightarrow 0^+} L(r) = 0$. 
By Osgood's lemma, we have
\begin{equation} \label{Osg}
-\mathcal{M} (|\phi(x,t) - \phi(y,t) | ) \le -\mathcal{M} (|x-y| ) + \int_0 ^t \|u(s) \|_L \rmd s,
\end{equation}
which guarantees uniqueness of solutions of the ODE \eqref{flowmap}.

We consider solutions to the transport equation \eqref{transport}.
Recall that for the continuity equation driven by an Osgood velocity field, all integrable weak solutions are Lagrangian and are therefore unique  \cite{AB08,CC21}.  Our aim here to show that certain singular structures are invariant in these solutions. We quantify these using
the following seminorm
\begin{equation}
[f]_{\gamma, L} = \lim_{r\rightarrow 0^+} \sup_{x, y: 0 < |x-y| < r} \frac{|f(x) - f(y) |}{\mathcal{M} (|x-y| )^\gamma }, \qquad  \gamma \in \mathbb{R},
\end{equation}
as well as the 
 localized seminorm
\begin{equation}
[f]_{x, \gamma, L} = \lim_{r\rightarrow 0^+} \sup_{ y: 0 < |x-y| < r} \frac{|f(x) - f(y) |}{\mathcal{M} (|x-y| )^\gamma }.
\end{equation}
 Our first result, inspired by the recent work of Chae and Jeong \cite{CJ21} on the preservation of logarithmic cusps for Lipschitz drifts, is:

\begin{theorem}\label{thmnormprop}
Let $L: (0,1) \rightarrow \mathbb{R}^+$ be an Osgood modulus of continuity as above, with $\mathcal{M}$ given by \eqref{mdef} and satisfying  $\lim_{z \rightarrow 0^+} \mathcal{M}(z) = \infty$.
Suppose that an (incompressible) velocity field $u$ has the modulus of continuity $L$, i.e. there exists a bounding function  $C\in L^1_{loc}(0,\infty)$ such that 
\begin{equation}
|u(x,t) - u(y,t) | \le C(t) L(|x-y|), \qquad \forall t \in \mathbb{R}^+.
\end{equation}
Then, for $\gamma\in \mathbb{R}$, the unique weak solution $\theta(t)$ of the transport equation \eqref{transport} satisfies
\begin{equation}
[\theta(t)]_{\gamma, L} = [\theta_0]_{\gamma, L}.
\end{equation}
In fact, the localized seminorm $[f]_{x, \gamma, L}$
is carried by the flow $\phi$  of $u$, i.e. 
$
[\theta(t)]_{\phi(x,t), \gamma, L} = [\theta_0]_{x, \gamma, L}.
$
\end{theorem}
Here, $\gamma>0$ corresponds to the propagation of singularity, and $\gamma<0$ corresponds to the propagation of cusps. 
\begin{remark}
More generally, if $G : (0, \infty) \rightarrow (0, \infty)$ is a monotone function satisfying 
\begin{equation}
\lim_{r\rightarrow \infty} \frac{G(r+c)}{G(r) } = 1,\qquad \text{for any} \  c \in \mathbb{R},
\end{equation}
for example, $G(x) = x^\gamma$ or $G(x) = \log (1+x)$, then the following local seminorm
\begin{equation}
[\theta]_{x, G, L} = \lim_{r\rightarrow 0^+} \sup_{y: 0 < |x-y| < r } \frac{|f(x) - f(y) | }{G(\mathcal{M} (|x-y| ) ) }
\end{equation}
is preserved by the flow: $[\theta(t) ]_{\phi(x,t), G, L} = [\theta_0]_{x, G, L}$.
\end{remark}
\begin{proof}
Define
\begin{equation}
R(z) := \exp ( - \mathcal{M}(z) ).
\end{equation}
Since $u$ is Osgood, it follows that $\lim_{z \rightarrow 0^+} R(z) = 0$. Moreover, $R$ is strictly increasing in $z$ and $R(1) = 1$. Next, by exponentiating \eqref{Osg}, we have
\begin{equation}
\frac{R(|\phi(x,t) - \phi(y,t) | )}{R(|x-y|) } \le \mu(t), 
\end{equation}
where 
\begin{equation}\label{mudef}
\mu(t) = e^{\int_0 ^t \| u(s) \|_L ds}.
\end{equation}
Here, $|x-y|$ is assumed to be small so that $\mu(t) R(|x-y|) < 1$. This assumption implies that $|\phi(x,s) - \phi(y,s) | < 1$ for all $s \in [0, t]$: first $s \rightarrow |\phi(x,s) - \phi(y,s) |$ is a continuous function of $s$. Let $t_0 = \sup \{ \tau \in [0, t] : |\phi(x,s) - \phi(y,s) | < 1 \text{ for all } s \le [0, \tau] \}.$ Suppose $t_0 < t$. Then for all $s < t_0$,
\begin{equation}
|\phi(x,s) - \phi(y,s) |  \le R^{-1} (\mu(s) R(|x-y|) ) \le R^{-1} (\mu(t) R(|x-y| ) ) < 1.
\end{equation}
Thus, $|\phi(x,t_0) - \phi(y,t_0) | \le R^{-1} (\mu(t) R(|x-y| ) ) < 1$ and also if we extend $s \rightarrow |\phi(x,s) - \phi(y,s) |$ past $t_0$, then there exists a small $\varepsilon>0$, $|\phi(x,s) - \phi(y,s) | < 1$ for every $s \in (t_0, t_0 + \varepsilon) \subset [0, t]$, which is a contradiction. Therefore, $t_0 = t$.

Also, we can reverse the time: then we obtain
\begin{equation} \label{Ballestimate}
\frac{1}{\mu(t) } \le \frac{R(|\phi(x,t) - \phi(y,t) | )}{R(|x-y|) } \le \mu(t).
\end{equation}
The consequence of \eqref{Ballestimate} is the following: let $r<1$ be a small radius so that $R(r) \mu(t), R(r) \mu(t)^2 < 1$. Now if $|x-y|<r$, then 
\begin{equation}
|\phi(x,t) - \phi(y,t) | \le R^{-1} (\mu(t) R(r)).
\end{equation}
Therefore, we have 
\begin{equation}
\phi( B_r (x), t) \subset B_{R^{-1} (\mu(t) R(r) ) } (\phi(x,t) ).
\end{equation}
This implies that 
\begin{equation}
\sup_{y \in B_r(x) } f(\phi(y,t)) \le \sup_{y \in B_{R^{-1} (\mu(t) R(r) ) }(\phi(x, t) )} f(y)
\end{equation}
On the other hand, suppose that $y \in B_r (\phi(x,t))$. Then there is $z \in \mathbb{R}^d$ such that $y = \phi(z,t)$. Thus, we have
\begin{equation}
r \ge | y - \phi(x,t) | = | \phi(z,t) - \phi(x,t) | \ge R^{-1} \left ( \frac{R(|z-x|)}{\mu(t) } \right ).
\end{equation}
In particular, we have
\begin{equation}
|z-x| \le R^{-1} (\mu(t) R(r) ).
\end{equation}
Therefore,
\begin{equation}
\sup_{y \in B_r (\phi(x,t))} f(y) \le \sup_{z: \phi(z,t) \in B_r(\phi(x,t)), z \in B_{R^{-1} (\mu(t) R(r) ) } (x) } f(\phi(z,t)) \le \sup_{z \in B_{R^{-1}(\mu(t) R(r) ) } (x) } f(\phi(z,t)).
\end{equation}
Consequently, we have
\begin{equation}
\sup_{y \in B_r (x) } f(\phi(y,t) ) \le \sup_{y \in B_{R^{-1} (\mu(t) R(r) ) } (\phi(x,t) ) } f(y) \le \sup_{y \in B_{R^{-1}  (\mu(t)^2  R(r) ) } (x) } f(\phi(y,t) ).
\end{equation}
Similarly, we can see the following:
\begin{equation}
\sup_{y \in B_{R^{-1} (\mu(t)^{-1} R(r) )} (\phi(x,t))} f(y) \le \sup_{y \in B_r (x) } f(\phi(y,t)) \le \sup_{y \in B_{R^{-1} (\mu(t) R(r) )} (\phi(x,t) )} f(y).
\end{equation}
Since $\lim_{r \rightarrow 0^+} R(r) = 0$, by taking limit $r \rightarrow 0^+$, radius of all balls in the above shrinks to 0 as $r \rightarrow 0^+$. Therefore, we have
\begin{equation}
\lim_{r \rightarrow 0^+} \sup_{y \in B_r (x) } f(\phi(y,t) ) = \lim_{r \rightarrow 0^+} \sup_{y \in B_r (\phi(x,t) ) } f(y).
\end{equation}
Another consequence of \eqref{Ballestimate} is the following: again by taking logarithm, we have
\begin{equation}
- \log \mu(t) - \mathcal{M} (|x-y| ) \le - \mathcal{M} (|\phi(x,t) - \phi(y,t) | ) \le \log \mu(t) - \mathcal{M} (|x-y| ).
\end{equation}
Dividing by $-\mathcal{M}(|x-y|)<0$ gives
\begin{equation}
1 + \frac{\log \mu(t) }{\mathcal{M}(|x-y| ) } \ge \frac{ \mathcal{M} (|\phi(x,t) - \phi(y,t) | )}{\mathcal{M}(|x-y| ) } \ge 1 -  \frac{\log \mu(t) }{\mathcal{M}(|x-y| ) }.
\end{equation}
Since $\lim_{r\rightarrow 0^+} \mathcal{M}(r) = \infty$, both leftmost and rightmost terms for the above are converging to $1$, and thus for any $\gamma \in \mathbb{R}$, we have 
\begin{equation}
\begin{gathered}
\lim_{r \rightarrow 0^+} \sup_{y \in B_r (x) } \frac{ \mathcal{M} (|\phi(x,t) - \phi(y,t) | )^\gamma }{\mathcal{M}(|x-y| )^\gamma } = 1, \\
\lim_{r \rightarrow 0^+} \inf_{y \in B_r (x) } \frac{ \mathcal{M} (|\phi(x,t) - \phi(y,t) | )^\gamma }{\mathcal{M}(|x-y| )^\gamma } = 1.
\end{gathered}
\end{equation}
A more quantitative version of the statement is the following: for any $y \in B_r(x)$, $\frac{ \mathcal{M} (|\phi(x,t) - \phi(y,t) | )^\gamma }{\mathcal{M}(|x-y| )^\gamma }$ lies between $ \left (1- \frac{\log \mu(t) }{\mathcal{M}(r) } \right )^\gamma$ and $ \left (1+ \frac{\log \mu(t) }{\mathcal{M}(r) } \right )^\gamma$.

In fact, we can generalize this to various functions $G$ specified above. We see that 
\be
\frac{G(\mathcal{M}(|x-y|) - \log \mu(t))}{G(\mathcal{M}(|x-y| ) )}\leq \frac{G(\mathcal{M}(|\phi(x,t) - \phi(y,t) | ) )}{G(\mathcal{M}(|x-y| ) )}\leq \frac{G(\mathcal{M}(|x-y|) + \log \mu(t))}{G(\mathcal{M}(|x-y| ) )}.
\ee
 As $|x-y| \rightarrow 0$, those two converges to $1$. Therefore, we have
\begin{equation}
1 = \lim_{r\rightarrow 0^+} \inf_{y \in B_r (x) } \frac{G(\mathcal{M} (|\phi(x,t) - \phi(y,t) | ) ) }{G(\mathcal{M}(|x-y|) ) } =  \lim_{r\rightarrow 0^+} \sup_{y \in B_r (x) } \frac{G(\mathcal{M} (|\phi(x,t) - \phi(y,t) | ) ) }{G(\mathcal{M}(|x-y|) ) }
\end{equation}

Therefore, we have the following:
\begin{equation}
\begin{gathered}
\lim_{r\rightarrow 0^+}\sup_{y \in B_r (x) } \frac{|\theta_0 (x) - \theta_0 (y) |}{\mathcal{M} (|x-y| )^\gamma } = \lim_{r\rightarrow 0^+} \sup_{y \in B_r (x) }\frac{|\theta(\phi(x,t), t) - \theta(\phi(y,t), t) |}{\mathcal{M}(|\phi(x,t) - \phi(y,t) | )^\gamma} \frac{\mathcal{M}(|\phi(x,t) - \phi(y,t) |)^\gamma}{\mathcal{M}(|x-y| )^\gamma} \\
=\lim_{r\rightarrow 0^+} \sup_{y \in B_r (x) }\frac{|\theta(\phi(x,t), t) - \theta(\phi(y,t), t) |}{\mathcal{M}(|\phi(x,t) - \phi(y,t) | )^\gamma} = \lim_{r\rightarrow 0^+} \sup_{y \in B_r (\phi(x,t) ) } \frac{|\theta(\phi(x,t),t) - \theta(y,t) | }{\mathcal{M} (|\phi(x,t) - y| )^\gamma }.
\end{gathered}
\end{equation}

More quantitatively, we can say that $\sup_{y \in B_r (x) } \frac{|\theta_0 (x) - \theta_0 (y) |}{\mathcal{M} (|x-y| )^\gamma }$ lies between $$ \sup_{y \in B_{R^{-1} (\mu(t)^{-1} R(r) ) } (\phi(x,t) ) } \frac{|\theta(\phi(x,t),t) - \theta(y,t) | }{\mathcal{M} (|\phi(x,t) - y| )^\gamma } \inf_{y \in B_r (x) } \frac{ \mathcal{M} (|\phi(x,t) - \phi(y,t) | )^\gamma }{\mathcal{M}(|x-y| )^\gamma }$$ and  $$\sup_{y \in B_{R^{-1} (\mu(t) R(r) ) } (\phi(x,t) ) } \frac{|\theta(\phi(x,t),t) - \theta(y,t) | }{\mathcal{M} (|\phi(x,t) - y| )^\gamma } \sup_{y \in B_r (x) } \frac{ \mathcal{M} (|\phi(x,t) - \phi(y,t) | )^\gamma }{\mathcal{M}(|x-y| )^\gamma }. $$
The claim follows.
\end{proof}

Theorem \ref{thmnormprop} shows quantitatively that a certain measure of regularity (or rather, singularity) is preserved in time.  We can go further to say that certain specific singular structures do not change their shape under evolution. We have

\begin{theorem}[propagation of local singular structure]\label{thmnormprop2}
There exists $r_0 \in (0,1)$, which depends only on the modulus of continuity $L$ and $T>0$, satisfying the following: suppose that
\begin{equation}\label{theta0data}
\theta_0 (x) = F (\mathcal{M} (|x - x_0 | )) + b_0(x)
\end{equation}
for $x \in B_r (x_0)$, $r<r_0$, with $\|b_0 \|_{L^\infty (B_r (x_0) )} \le B_0$ and $F$ smooth with  
\begin{equation}
[F] = \sup_{|z| \ge 1} |F'(z) | < \infty.
\end{equation}
Then for $ t \in [0, T]$ and $x \in B_{R^{-1} (\mu(t) ^{-1} R(r) )} (\phi(x_0, t) ) $, the solution $\theta(x,t)$ of the transport equation  \eqref{transport} takes the form
\begin{equation}\label{thetasolution}
\theta(x,t) = F(\mathcal{M} (|x - \phi(x_0, t) | ) ) + b(x,t),
\end{equation} 
where the remainder is a bounded function satisfying
\begin{equation}
\| b(\cdot, t)\|_{L^\infty (B_{R^{-1} (\mu(t)^{-1} R(r) ) } (\phi(x_0, t) ) )} \le \|b_0 \|_{L^\infty (B_r (x_0 ) ) } + [F] \log \mu(t)
\end{equation}
where $\mu$ is defined by equation \eqref{mudef}.
\end{theorem}
We remark that, the size of the corrector $b$ does not explicitly depend on the radius $r$ (other than that of $b_0$.) The rate of loss of radius for the local expression depends on $\mu(t)$ and the modulus of continuity $L$.
\begin{proof}
We note that
\begin{align*}
\theta(x,t) &= F(\mathcal{M} (|\phi_t ^{-1} (x) - x_0 | ) ) + b_0 (\phi_t ^{-1} (x) ) \\
&= F(\mathcal{M} (|x - \phi(x_0, t) | ) )+ b_0 (\phi_t ^{-1} (x) )   \\
&\qquad+ \big (F(\mathcal{M} (|\phi_t ^{-1} (x) - x_0 | ) ) - F(\mathcal{M} (|x - \phi(x_0, t) | ) )  \big ) 
\end{align*}
and recall that if $x \in B_{R^{-1} (\mu(t)^{-1} R (r) )} (\phi(x_0,t) )$, then $\phi_t ^{-1} (x) \in B_r (x_0)$. Then we have 
\begin{equation} \label{bLag}
b(x,t) = \big (F(\mathcal{M} (|\phi_t ^{-1} (x) - x_0 | ) ) - F(\mathcal{M} (|x - \phi(x_0, t) | ) )  \big ) + b_0 (\phi_t ^{-1} (x) )
\end{equation}
and 
\begin{equation}
|b(x,t) | \le  \| b_0 \|_{L^\infty} +  \sup_{\lambda \in [0,1] } | F' ( \lambda \mathcal{M} (|\phi^{-1}_t (x) - x_0| ) + (1-\lambda) \mathcal{M}(|x - \phi(x_0, t) |  ) ) | \log \mu(t)
\end{equation}
for $x \in B_{R^{-1} (\mu(t)^{-1} R(r) ) } (\phi(x_0, t) )$. Finally, if $r_0$ is chosen small so that both $\mathcal{M}(r_0)>1$ and $\mathcal{M}(R^{-1} (\mu(T) R(r_0) ) ) > 1$, then we see that $\mathcal{M} (|\phi^{-1}_t (x) - x_0| ), \mathcal{M} (|x - \phi(x_0, t) | ) > 1$ for all $t \in [0,T]$. Thus, we see that
\begin{equation}
|b(x,t) | \le \|b_0 \|_{L^\infty} + [F] \log \mu(t).
\end{equation}
The claim follows.
\end{proof}

\begin{remark}[Types of propagated singularities]
In particular, singularities of the form $\mathcal{M} (|x - x_0|)^\gamma$, where $\gamma \le 1$, or $\log \mathcal{M} (|x- x_0| )$, propagates up to $L^\infty$ correction. Furthermore, this suggests that various pathologies around a singularity point may propagate over the flow: for example, fix $T>0$ and consider \eqref{theta0data} 
with $F$ smooth but oscillating: for example, $F(z) = \sin (\lambda z)$, where $\lambda$ is so small that $\log \mu(T) \lambda < \frac{1}{2}$ and thus $|b(x,t) | \le \frac{1}{2}$. Now as $|x - \phi(0, t) | \rightarrow 0$, $\mathcal{M} (|x - \phi(0, t) | ) \rightarrow \infty$ and thus $\theta(x,t)$ changes sign like Topologist's sine curve as $x$ approaches $\phi(0,t)$, up to time $0 \le t \le T$. 
\end{remark}

\subsection{Sharpness of the linear results}\label{sharp}
We note that Theorem \ref{thmnormprop2} is sharp in the following sense: if $F(z)$ is allowed to grow super-linearly in $z$, then the structures \eqref{theta0data} will not be propagated for some Osgood flows.
 To see this, from \eqref{bLag}, it suffices to estimate
\be
F(\mathcal{M} ( |\phi_t ^{-1} (x) - x_0| ) - F(\mathcal{M} (|x - \phi(x_0, t) | ).
\ee
Here, suppose that $x_0=0$ is a stagnation point, so that $\phi(x_0, t) = 0$ for all $t\in \mathbb{R}$.  Suppose also $F$ is superlinear so that for any $C>0$, there exists $\varepsilon>0$ such that 
\be
F'(\mathcal{M}(r) ) \ge C 
\ee
for any $0<r<\varepsilon$. Also denote $\phi_t ^{-1} (x) = y$ so that $x = \phi(y, t) = \phi_t (y)$. Then for arbitrary $C>0$ and $|y|$ sufficiently close to $0$ accordingly, we have
\begin{align}\nonumber
F(\mathcal{M} ( |\phi_t ^{-1} (x) - x_0| )) - F(\mathcal{M} (|x - \phi(x_0, t) | )) &= \int_{|\phi_t (y) | } ^{|y|} \frac{F'(\mathcal{M}(r) )}{L(r) } \rmd r \\
& \ge C \int_{|\phi_t (y) | } ^{|y| } \frac{ \rmd r } {L(r) } = C\Big( \mathcal{M} (|\phi_t (y) | ) - \mathcal{M} (|y| ) \Big).\label{fmbd}
\end{align}
Here are two cases:
\begin{enumerate}
\item \emph{(Lipschitz setting)}
Consider the 2D hyperbolic flow $u(x,y) = (y,x)$.  This vector field is Lipschitz, and thus $L(z)=z$ and $\mathcal{M}(z) = \log (1/z)$.   
Here, one may choose $y$ so that $\phi_t (y) = e^{-t} y.$ For such $y$, we have
\be \label{Lipcond}
\mathcal{M} (|\phi_t (y) | ) - \mathcal{M} (|y| ) = t.
\ee
Thus, for any $C>0$ and $t>0$, one can choose $x$ so that $F(\mathcal{M} ( |\phi_t ^{-1} (x) - x_0| ) - F(\mathcal{M} (|x - \phi(x_0, t) | ) \ge Ct$.  It follows that for any $t\neq 0$, the remainder is infinite $\|b(t) \|_{L^\infty} = \infty$.

\item \emph{(log-Lipschitz setting)}
Here $L(z) = z \log (1/z)$ and $\mathcal{M} (z) = \log\log(z)$. Thus, we have 
\be
\mathcal{M} (|\phi_t (y) | ) - \mathcal{M} (|y| ) = \log\frac{\log |\phi_t (y) |}{\log|y| }.
\ee
In order to obtain a counterpart of \eqref{Lipcond}, we need $\frac{\log |\phi_t (y) | }{\log |y| } \ge C(t)$ independent of $y$. On can choose the Bahouri-Chemin vortex with velocity field $u=\nabla^\perp \psi$ where $\psi$ solves $\Delta \psi  = {\rm sgn}(x){\rm sgn} (y)$ on domain $\mathbb{T}^2$, see e.g. \cite{BCD11,D15}.  The $y$-coordinate axis is invariant for the correspond flow, and there exists $0<C_1<C_2<\infty$ such that
\be
C_1 |y|^{e^{-t}} \leq |\phi_t (y) | \le C_2 |y|^{e^{-t}},
\ee
so we have
\be
\frac{\log |C_2|}{\log |y|}+ e^{-t}   \leq  \frac{\log |\phi_t (y) |}{\log|y| } \leq \frac{\log |C_1|}{\log |y|} + e^{-t} .
\ee
It follows that for all $|y|$ sufficiently small (depending on $t$), we have
\be
\tfrac{1}{2} e^{-t}   \leq  \frac{\log |\phi_t (y) |}{\log|y| } \leq 2 e^{-t}. 
\ee
As such, for such $y$ we have that 
\be
\mathcal{M} (|\phi_t (y) | ) - \mathcal{M} (|y| )>\tfrac{1}{2} t.
\ee
In view of \eqref{fmbd}, the propagation of singularities result is sharp for log-Lipschitz vector fiel\rmd s.  
\end{enumerate}
Similar examples can be used to show the sharpness of our Theorem for a general Osgood modulus of continuity. 

\section{Propagation of Singularities in 2D Euler}\label{seceuler}

We now establish the main application of these ideas to nonlinear transport equations. We focus on the 2D Euler equation, governing the motion of a perfect (inviscid and incompressible) fluid, set on domains without boundary for simplicity, i.e. $\mathbb{R}^2$ or $\mathbb{T}^2$.  In vorticity form these read \eqref{ee1}--\eqref{ee2}.
Here we prove the propagation of  a $\log\log_+ (1/|x|)$ vortex (along with slightly more mild structures) to bounded perturbations.  This results allows for the derivation of a system of loglog singular vortices interacting. This system is akin to the vortex-wave system for point vortices \cite{MP91}.  Using the ideas in  \cite{CDE2019,CCS21}, one can show that such solutions  quantitatively (with a rate of convergence) stable in inviscid limit from the Navier-Stokes equation.

\begin{remark}
We remark that, philosophically, the roles of $L$ and $\mathcal{M}$ have reversed passing from \S \ref{lintrans} to \S \ref{seceuler}.
 In \S \ref{lintrans}, we specified a modulus $L$ for the velocity field and studied the singularities $\mathcal{M}$ (related to $L$ via $\mathcal{M}'(z) = -{1}/{L(z)}$) that are propagated.  In \S \ref{seceuler}, we specify a singular structure $\mathcal{M}$ at the level of the vorticity, which induces some Osgood velocity, and propagate it up to lower-order remainders.   The key point is that the Osgood modulus induced by the Biot-Savart law does not correspond to $L(z) = -1/\mathcal{M}'(z)$ (it is, in general, worse) and thus the results from the linear theory do not directly apply.  Instead, some cancellations occur in the nonlinear evolution equation which permits the propagation of the singular structure but generates a weak background field.
\end{remark}

We will work with the generalized Yudovich spaces $\omega_0\in Y_\Theta$:
\be\label{Ytheta}
Y_\Theta := \left\{ f \in \bigcap_{p\in [1,\infty)} L^p \ : \ \|f\|_{Y_\Theta} = \sup_{p\in [1,\infty)} \frac{ \|f\|_{L^p}}{\Theta(p)} <\infty\right\},
\ee  
in which there exist a \emph{unique} weak solution to the Euler equation \cite{Y95} provided 
\be\label{thetaint}
\int^{+\infty} \frac{\rmd p}{p \Theta (p)} = +\infty.
\ee
Having $\omega\in Y_\Theta$ induces by the Biot-Savart law an \emph{Osgood continuous} velocity field $u=K[\omega]$ enjoying the modulus of continuity: for $|x-y|$ sufficiently small,
\be\label{urbd}
|u(x,t)-u(y,t)| \lesssim  C|x-y| \log(1/|x-y|) \psi_\Theta(|x-y|)
\ee
where  $ \psi_\Theta(z) = \Theta(\log (1/z))$.
See Theorem 1.6 in \cite{CS21}.

\subsection{Propagation of $\log\log(1/|x|)$ vortices}

Let $L: (0,1) \rightarrow \mathbb{R}^+$ be any Osgood modulus of continuity satisfying 
\be \label{continuityconstraint}
 z \log(1/z) \lesssim  L(z).
\ee
Let  $\mathcal{M}$ be corresponding to $L$ and given by \eqref{mdef}.  For example, we can have 
\be
\mathcal{M}(z)= \log_{1+n}(1/z) \qquad \text{ for any } \qquad n\in \mathbb{N}.
\ee
The most singular vortex consider corresponds to $\mathcal{M}(z)=\log\log_+(1/z)$.   See Fig. \ref{logvort} for a cartoon.

First note that the vorticity of the form $\mathcal{M} (|x | )$ is in the generalized Yudovich space \eqref{Ytheta}  $Y_\Theta$ for $\Theta (p)= \log p$.  To see this, note that since \eqref{continuityconstraint} holds, we have
\be
\mathcal{M}(z) = \int_{z} ^{\mathsf{m}_L} \frac{\rmd r}{L(r) } \lesssim \int_{z} ^{\mathsf{m}_L} \frac{\rmd r}{r \log (1/r) } = \int_{\log (1/\mathsf{m}_L )} ^{\log (1/z)} \frac{\rmd s}{s} = \log \log (1/z) - \log \log (1/\mathsf{m}_L ).
\ee
However, it is known that $x \rightarrow g(x) := \log \log (1/|x|) \chi_{B_1 (0) } (x)$ has 
\be
\| g \|_{L^p  (B_1)} \sim \log p
\ee
for large $p$ (see Example 3.3 of Yudovich's paper \cite{Y95}), so $\mathcal{M}(|x|)$ belongs to the generalized Yudovich space with $\Theta(p) = \log p$.  See Lemma \ref{lemmayudo} herein for a slightly more general statement.

\begin{figure}[h!]
\centering
\includegraphics[scale=.3]{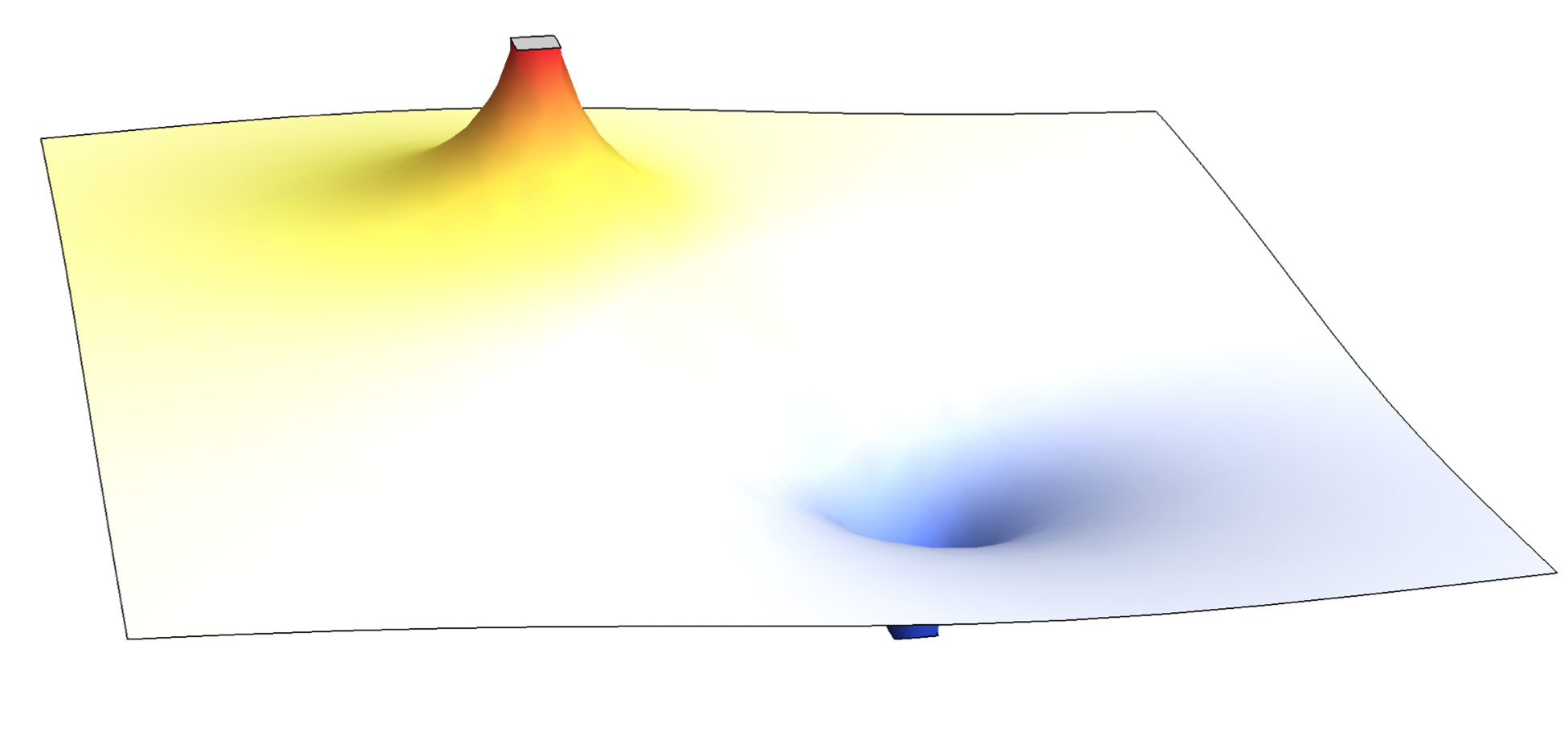}
\caption{Cartoon of double-logarithmically singular vortices.}
\label{logvort}
\end{figure}  

We prove that $\mathcal{M}$--type singularities in the vorticity are preserved under Euler evolution.

\begin{theorem}[propagation of singularities for 2D Euler]\label{2dEpropsing}
Let $\Theta(p)= \log_k(p)$ for $k\geq 0$. Consider  $b_0\in Y_{\Theta}(\mathbb{R}^2)\cap L^1(\mathbb{R}^2)$, forcing $f\in L^1_{loc}(\mathbb{R}; Y_{\Theta}(\mathbb{R}^2)\cap L^1(\mathbb{R}^2))$  and  parameters $L$ and $\mathcal{M}$ be as above.  Consider initial data of the form
\begin{equation}\label{om0data}
\omega_0 (x) = \mathcal{M} (|x | ) + b_0(x).
\end{equation}
Then, the unique weak solution $\omega(t)$ of 2D Euler equations \eqref{ee1}--\eqref{ee2} with forcing $f$ emanating from $\omega_0$
 takes the following form: there exists a scalar function $b \in L_{loc}^\infty(\mathbb{R}; Y_{\Theta}\cap L^1(\mathbb{R}^2))$  and a map $\phi_*(t): \mathbb{R}\to \mathbb{R}^2$ such that 
\begin{equation}\label{omsolution}
\omega(x,t) = \mathcal{M} (|x - \phi_*(t) |  ) + b(x,t).
\end{equation}
Thus, the  singularity propagates over time, up to a bounded correction.
\end{theorem}

\begin{proof}
Note that $\Theta(p)= \log_k(p)$ satisfies the condition \eqref{thetaint}.
We decompose this solution as
\be\label{omform}
\omega(x,t) = \omega_{\textsf{s}} (x,t) + b(x,t),\qquad \omega_{\textsf{s}} (x,t) :=\mathcal{M} (|x-\phi_*(t)|)
\ee
where $b$ and $\phi_*$ are to be determined.  We aim to show that   $b(t)\in Y_\Theta \cap L^1$ for all $t>0$.  If established and we let $u_{\textsf{r}}:=\nabla^\perp \Delta^{-1} b$, then according to the estimate \eqref{urbd} is Osgood continuous.  The map $\phi_*(t)$ will denote the position of particle at time $t$ which was originally located at $0$ and flowed by the regular part of the velocity field $u_{\textsf{r}}$.   We will work under the assumption that $b$ is bounded which, as we will see, can be bootstrapped. The full velocity field has contribution also from the singular part $u_{\textsf{s}} =\nabla^\perp \Delta^{-1}  \omega_{\textsf{s}} $:
$$
u(x,t) =u_{\textsf{s}} (x,t)  + u_{\textsf{r}} (x,t),
$$
 but since $ \omega_{\textsf{s}}\in Y_{\log(p)}$, the velocity field $u_{\textsf{s}} (x,t)$ has modulus of continuity $r \mapsto r \log (e/r) \log \log (e/r)$ for small $r>0$ and thus $u$ is Osgood continuous (see e.g. \cite{CS21}).

We now show that the above picture hol\rmd s. Inserting the ansatz \eqref{omform} into  \eqref{ee1}--\eqref{ee2} we obtain
$$
f=(\partial_t + (u_{\textsf{s}}+ u_{\textsf{r}}) \cdot \nabla) \omega(x,t)  = (\partial_t +(u_{\textsf{s}}+ u_{\textsf{r}}) \cdot \nabla) b(x,t) + (\partial_t +(u_{\textsf{s}}+ u_{\textsf{r}})  \cdot \nabla) \omega_{\textsf{s}}.
$$
Now note that, since $\omega_{\textsf{s}}(x,t)$ is a radial function, its associated velocity $u_{\textsf{s}}$ is a circular flow, i.e. $u_{\textsf{s}}(x,t)\propto x^\perp$.  It follows that there is no self-interaction $u_{\textsf{s}}\cdot\nabla \omega_{\textsf{s}}=0$.  Thus, we have
\begin{align}\nonumber
(\partial_t + (u_{\textsf{s}}+ u_{\textsf{r}})  \cdot \nabla)\omega_{\textsf{s}}&= (\partial_t + u_{\textsf{r}}  \cdot \nabla) \omega_{\textsf{s}}\\\nonumber
 &= \frac{(x-\phi_*(t))\cdot (\partial_t + u_{\textsf{r}} \cdot \nabla)( x-\phi_*(t)) }{|x-\phi_*(t)|}\mathcal{M}'(|x-y(t)|)\\
&=  \frac{(x-\phi_*(t))\cdot (- \dot{\phi_*}(t)+ u_{\textsf{r}}(x,t)) }{|x-\phi_*(t)|}\mathcal{M}'(|x-y(t)|).
\end{align}
Since  $u_{\textsf{r}}$ is Osgood, we can define the map $\phi_*:\mathbb{R}\to \mathbb{R}^2$ to be the unique solution of 
\be
\frac{\rmd}{\rmd t} {\phi_*}(t) = u_{\textsf{r}}({\phi_*}(t),t), \qquad \phi_*(0)=0.
\ee
Thus we obtain the following equation for the remainder $b$:
\be\label{remaindereqn}
(\partial_t +  (u_{\textsf{s}}+ u_{\textsf{r}})  \cdot \nabla) b(x,t) =\frac{(x-{\phi_*}(t))\cdot ( u_{\textsf{r}}({\phi_*}(t),t)- u_{\textsf{r}}(x,t)) }{|x-{\phi_*}(t)|}\mathcal{M}'(|x-{\phi_*}(t)|) + f =: g(x,t).
\ee
The equation for the remainder \eqref{remaindereqn} has a unique weak solution and we express the solution by
\be
b(x,t) = b_0 (\phi_t ^{-1} (x)) + \int_0 ^t g (\phi_s (\phi_t ^{-1} (x) ), s) \rmd s.
\ee
Since $\omega_0\in L^1$ and $\mathcal{M}(z)\in L^1$, the solution satisfies the estimate
\be
\|b(t)\|_{L^1} \leq C + \|b_0\|_{L^1}.
\ee
Denoting $r_{s,t,x} := |\phi_s (\phi_t ^{-1} (x) ) - \phi_* (s) |$ 
and recalling from \eqref{urbd}  the  modulus of continuity for $u_{\textsf{r}}$ is
\be
L_b(z):= Cz \log(1/z)  \Theta( \log (1/z)),
\ee
 we have 
\be\label{gbound}
|g(\phi_s (\phi_t ^{-1} (x)) ,s) | \le \| u_{\textsf{r}} \|_{L_b} L_b (r_{s,t,x} ) |\mathcal{M} ' (r_{s,t,x} )|.
\ee
Note also that our assumption \eqref{continuityconstraint} on $L$ guarantees that 
\be\label{bndM}
z \log(1/z)  |\mathcal{M}'(z)| \lesssim 1.
\ee
This implies that $|g(\phi_s (\phi_t^{-1} (x)), s) | \lesssim  \Theta(\log (1/r_{s,t,x}))$. Finally, we compute $\| b \|_{L^p}$.
\begin{align}\nonumber
\| b (t)  \|_{L^p } &\le \| b_0 \circ \phi_t ^{-1} \|_{L^p}  + \int_0 ^t \| \Theta(\log(1/ r_{s,t,x})) \|_{L^p} \rmd s\\
&\le \| b_0 \|_{L^p}  + t \| \Theta( \log (1/|\cdot|)) \|_{L^p}\label{bound}
\end{align}
by volume preservation.   Note that for $\Theta(p)= \log_k(p)$ for $k\geq 1$, we have $\| \Theta( \log (1/|\cdot|)) \|_{L^p}= \ \log_{k+1}(p) \|_{L^p}=\log_k(p)= \Theta(p)$ by Lemma \ref{lemmayudo}.  If $k=0$ and $\Theta(p)=1$, then we control the $L^\infty$ norm. Since  $\| b_0 \|_{L^p}\leq \Theta(p)$  and $\| \Theta( \log (1/|\cdot|)) \|_{L^p}\leq \Theta(p)$, the assumption on remainder bounded in a suitable Yudovich space can be bootstrapped and the claim is proved.
\end{proof}

We now prove the Lemma to show that the $\log_k$ are in Yudovich classes:
\begin{lemma}\label{lemmayudo}
For $n\geq 1$, the $L^p$ norms of iterated logarithmic singularities grow as follows:
\be
\|\log_n(\cdot)\|_{L^p} \lesssim \log_{n-1}(p)
\ee
where $\log_0(p)=p$ by convention.
\end{lemma}
\begin{proof}
To see this, note first that for all $x\in (0,1)$, we have that 
\be
|\log(x)| \leq n x^{-\frac{1}{n}}
\ee
 for any $n>0$.  This follows simply from the fact that it holds if and only if $y \leq n e^{y/n}$, which can be verified by considering the Taylor expansion of the exponential.  This  immediately shows that 
 \begin{align}
 \| \log(1/|\cdot|)\|_{L^p(B_1(0))}^p  &= 2 \pi \int_0^1 |\log(r)|^p r \rmd r  \leq 2 \pi  n^p \int_0^1 r^{1-\frac{p}{n}}\rmd r =\frac{2 \pi  n^{p+1}}{2n-p} 
 \end{align}
 provided $2n> p$.   The bound is minimized for $n= (1+p)/2$ so that $ \| \log(1/|\cdot|)\|_{L^p(B_1(0))}\approx  p$ as $p\to \infty$. Next we consider  $\log\log$.  Since if  $x\in (0,e^{-1})$, then $1/\log(x) \in (0,1)$, we have
 \begin{align}\nonumber
|\log|\log(x)|| &\leq |\log n| \left(\frac{1}{|\log x|}\right)^{-\frac{1}{|\log n|}} =  |\log n| |\log x|^{\frac{1}{|\log n|}} \leq e |\log n| x^{{-\frac{1}{n|\log n|}}}
 \end{align}
where we used the identity $n^{\frac{1}{\log n}}=e$.  Now, for $x\in (0,1)$, using that $n\log(n)> n-1$ we have $x^{{-\frac{1}{n|\log n|}}}\leq x^{{-\frac{1}{n-1}}}$.  It follows from the above calculation that
\be
\|\log|\log(\cdot)|\|_{L^p}   \lesssim \log p.
\ee
Now for the triple log, we have
 \begin{align}\nonumber
|\log|\log|\log(x)||| &\leq  |\log_2 n| |\log_2x|^{\frac{1}{|\log_2 n|}} \leq e^{\frac{1}{|\log_2 n|}} |\log_2 n| x^{{-\frac{1}{n|\log n||\log_2 n|}}}
 \end{align}
 using  $\log(n)^{\frac{1}{\log_2 n}}=e$.   Now, for $x\in (0,1)$, using that $n\log(n)\log_2(n)> n-1$ for $n$ sufficiently large, we have $x^{{-\frac{1}{n|\log n||\log_2 n|}}}\leq x^{{-\frac{1}{n-1}}}$.  Thus
 \be
\|\log|\log|\log(\cdot)||\|_{L^p}   \lesssim \log_2 p.
\ee
The pattern is straightforward to continue.
\end{proof}

\begin{remark}[Breakdown of $\log_2(1/|x|)\log_3 (1/|x|)$ vortices?]\label{logloglogloglogrem}
In the proof of Theorem \ref{2dEpropsing}, the $\log_2$ vortex appears to be critical in the sense that we do not expect slightly more singular vortices to propagate with bounded remainders.  Indeed, note that the vortex were any more singular, then \eqref{bndM} would be slowly diverging with $z$ and the integral in the remainder may be divergent. To have a  specific  example, consider data of the form
\be\label{loglogloglogloglog}
\omega_0(x) =\log_2(1/|x|)\log_3(1/|x|) + b_0(x)
\ee
with  $b_0\in Y_{\log_k(p)}$ for some $k\geq 0$ as in Theorem \ref{2dEpropsing}.
Note first that $\log_2(1/|x|)\in Y_\Theta$ for $\Theta(p):=\log(p) \log_2(p)$. Indeed,
since $\|f g\|_{L^p} \leq \|f \|_{L^{2p}}\| g\|_{L^{2p}}$,
\be
\|\log_2(\cdot )\log_3(\cdot)\|_{L^p}  \leq \|\log_2(\cdot )|\|_{L^{2p}} \|\log_3(\cdot )|\|_{L^{2p}}  \lesssim \log (2p) \log_2 (2p),
\ee
by Lemma \ref{lemmayudo}.
 As such, according to Yudovich's theory \cite{Y95}, there exists a global unique weak solution of 2D Euler living in the space $Y_\Theta$. 
In this case, \eqref{bndM} becomes
\be\label{bndM2}
z \log(1/z)  |\mathcal{M}'(z)| = 1+ \log_3(1/z) \leq 2\log_3(1/z),
\ee
so that in the remainder estimate \eqref{bound} we must compute
\begin{align}
 \| \log_3(1/|\cdot|)\Theta( \log (1/|\cdot|)) \|_{L^p}&=  \| \log_3(1/|\cdot|) \log_{k+1}(1/|\cdot|) \|_{L^p}
\end{align}
By Lemma \ref{lemmayudo}, this is bounded by
\begin{align}
\| \log_3(1/|\cdot|) \log_{k+1}(1/|\cdot|) \|_{L^p } &\leq \| \log_3(1/|\cdot|)\|_{L^{2p} } \| \log_{k+1}(1/|\cdot|) \|_{L^{2p} } \leq \log_2(p)\log_k(p).
\end{align}
Thus, we cannot conclude that $b(t)\in Y_{\log_k(p)}$ as we did Theorem \ref{2dEpropsing} since we lose an extra factor of $ \log_2(p)$ in the above estimate.

One may try to come up with counter examples for propagation.  For example, take
 $b_0\in L^\infty$ (i.e. $b_0\in Y_{\log_k(p)}$ with $k=0$) to be the Bahouri-Chemin vorticity field:
\be\label{BCv}
b_0(x_1,x_2) = {\rm sgn}(x_1) {\rm sgn} (x_2).
\ee 
It would be very interesting to show that this type of singularity does not propagate, i.e. that the solution does not take the form
\be
\omega(x,t) \neq \log_2(1/|x-\phi_*(t) |)\log_3(1/|x-\phi_*(t)|) + b(x,t)
\ee
with a bounded $b(x,t)$. This is indeed true at the linear level, since the Bahouri-Chemin vortex \eqref{BCv} generates a log-Lipshitz flow with double exponential attraction and repulsion on the coordinate axes.  The result then follows by the result of \S \ref{sharp}.
The difficulty is that there is a stabilizing effect of the nonlinearity where the radial vortex mixes the background and, in doing so, weakens the stretching effects of the hyperbolic background.  This gives hope that there could be some ``supercritical" uniqueness results for radial vortices with sufficiently  strong singularities (perhaps even  logarithmically divergent vorticity).  However, we believe that \eqref{loglogloglogloglog} is too weak for this effect to kick in and, as such, this singular structure should breakdown immediately.
\end{remark}

We finish this section by noting that one can also propagate any number of these $\log\log$ vortices as a system akin to that of vortex-wave \cite{MP91}.  Relative to the vortex-wave system, the system involving $\log\log$ vortices has the advantage of being a true (weak) solution of the Euler equations. On the other hand, the $\log\log$ solutions general non-trivial background velocities over time.
Specifically, stating the result for bounded perturbations for simplicity, we have

\begin{theorem}[propagation of singularities for 2D Euler]\label{2dEpropsing}
Let $b_0\in L^\infty(\mathbb{R}^2)\cap L^1(\mathbb{R}^2)$ and  $L$ and $\mathcal{M}$ be as above.  Let $N\in \mathbb{N}$ and $\{ \gamma_i\}_{i=1}^N$ with $\gamma_i\in \mathbb{R}$ and $\{ x_i \}_{i=1}^N$ with $x_i\in \mathbb{R}^2$ be given.  Consider initial data of the form \begin{equation}\label{om0data2}
\omega_0 (x) =\sum_{i=1}^N \gamma_i  \mathcal{M} (|x-x_0^i | ) + b_0(x),
\end{equation}
Then, the unique weak solution $\omega(t)$ of 2D Euler emanating from $\omega_0$ takes the following form: there exists a scalar function $b \in L_{loc}^\infty(\mathbb{R}; L^\infty(\mathbb{R}^2)\cap L^1(\mathbb{R}^2))$  and  a collection of maps $\phi_i(t): \mathbb{R}\to \mathbb{R}^2$ for $i=1, \dots, N$ such that 
\begin{equation}\label{omsolution2}
\omega(x,t) =\sum_{i=1}^N \gamma_i\mathcal{M} (|x-\phi_i(t) | )  + b(x,t)
\end{equation}
where the each $\phi_j$ corresponds to the Lagrangian flow of the velocity field excluding self-interaction
\be
\frac{\rmd}{\rmd t} \phi_j(t) = K\Big[\sum_{i\neq j} \gamma_i\mathcal{M} (|x-\phi_i(t) | )  + b(x,t)\Big]\circ \phi_j(t), \qquad \phi_j(0)= x_j,
\ee
where $K=\nabla^\perp \Delta^{-1}$ is the Biot-Savart law.
Thus, the  singularities propagate over time and, to leading order, interact with each other up to a bounded correction.
\end{theorem}

\begin{proof}
The proof is same as before (we set $f=0$ for simplicity, the same proof works with non-trivial forcing): again we let $b(x,t) = \omega (x,t) - \sum_{i=1} ^N \gamma_i \mathcal{M} (|x-\phi_i (t) |)$ and estimate the evolution equation of $b$. For simplicity, we denote \be \omega_i := \gamma_i \mathcal{M} (|x-\phi_i (t) | ), \quad u_i := K[\omega_i ],  \quad u_b := K[b]. \ee Since $u_i \cdot \nabla \omega_i = 0$ by the radial nature of $\omega_i$, we have
\begin{align*}
 &(\partial_t + u \cdot \nabla) b = -\sum_{i=1} ^N ( \partial_t + (u_b + \sum_{j \ne i} u_j ) \cdot \nabla ) \omega_i \\ &= -\sum_{i=1} ^N \gamma_i \mathcal{M}'(|x-\phi_i (t) | ) \frac{(x - \phi_i (t) )}{|x-\phi_i (t) | } \cdot (-\frac{\rmd}{\rmd t} \phi_i(t) ) + \gamma_i \mathcal{M}'(|x-\phi_i (t) | ) \frac{(x - \phi_i (t) )}{|x-\phi_i (t) | } \cdot (u_b + \sum_{j\ne i} u_j ) \\ &= \sum_{i=1} ^N \gamma_i \frac{1}{L(|x-\phi_i (t) | ) }  \frac{(x - \phi_i (t) )}{|x-\phi_i (t) | } \cdot \left( \sum_{j\ne i} (u_j (\phi_i (t), t) - u_j (x, t) ) + u_b (\phi_i (t), t) - u_b (x, t) \right) =: g(x, t).
 \end{align*}
We have the following bound on the forcing $g$ of the perturbation $b$:
\be\nonumber |g(x,t) | \le \sum_{i=1} ^N \gamma_i \| K[\mathcal{M} (|\cdot | )] \|_{L} + \sum_{i=1} ^N \gamma_i \| b\|_{L^1 \cap L^\infty} \frac{ |x- \phi_i (t)| \log(1/|x- \phi_i (t)|)}{L(|x-\phi_i (t) | ) } \le C \sum_{i=1} ^N (1 + \| b\|_{L^1 \cap L^\infty } ). \ee 
  As before, we have  $\| b(t) \|_{L^1} \le C + \|b_0 \|_{L^1} $ for some explicit $C>0$ and thus we once again obtain  \be \| b (t) \|_{L^\infty} \le \| b_0 \|_{L^\infty} + t (C + \| b_0 \|_{L^1} ) + C\int_0 ^t \| b(s) \|_{L^\infty} \rmd s. \ee By Gronwall's inequality, $b$ is bounded for any finite time $t \in \mathbb{R}$.
\end{proof}

\begin{remark}
Having established this result, there are a number of interesting dynamical questions that could be asked. For example, under what conditions do singular points remain in a fixed compact set of space for all time? This is related to the problem of whether vorticity which is initially compactly supported and nonnegative can have unbounded support as $t\to\infty$. 
\end{remark}

\subsection{Singularities determine the solution in the DiPerna-Majda class}\label{DMclass}

It is well known from the work of DiPerna and Majda \cite{DM87} that there exist global weak solutions in the class $\omega_0\in L^p$ for $p\geq1$, but these are not known to be unique.  In fact, if a non-autonomous force is included, Vishik proved non-uniqueness in this class for all $p\in [1,\infty)$ \cite{Vishik1,Vishik2} (see also \cite{notesVishik}).  Here we prove that, in a sense, the behavior of the singular structure determines the solution uniquely, provided $p\geq 2$. 

\begin{theorem}
Let $u_1, u_2\in C([0,T); W^{1,p}(\mathbb{R}^2))$  be two weak solutions to 2D Euler.  Assume 
\be
\|u_1-u_2\|_{\mathcal{G}^{{1}/{s}}}< \infty, \qquad \text{where} \quad s=2/p,
\ee
where $\mathcal{G}^{{1}/{s}}$ is the Gevrey space of index $s$. Then 
\be
\|u_1(t)-u_2(t)\|_{L^2}\leq \|u_1(0)-u_2(0)\|_{L^2}^{e^{- K_{*} \| \nabla u_1 \|_{L^p} t}},
\ee
where $K_*$ depends on the Gevrey norm of $u_1 - u_2$.
\end{theorem}

\begin{proof}
First note, by a simple energy estimate and the Sobolev embedding theorem,
\begin{align}\nonumber
\frac{1}{2}\frac{\rmd}{\rmd t}\|u_1-u_2\|_{L^2}^2 &\leq \int |u_1-u_2|^2 |\nabla u_1| \rmd x \leq  \|u_1-u_2\|_{L^{2p_*}}^2 \|\nabla u_1\|_{L^p}\leq C  \|u_1-u_2\|_{H^{1/p}}^2  \|\nabla u_1\|_{L^p}\\
&\leq C\left( \|u_1-u_2\|_{L^2}^2+  \|u_1-u_2\|_{\dot{H}^{1/p}}^2\right)  \|\nabla u_1\|_{L^p}.
\end{align}
Now observe that $\|f\|_{\dot{H}^s}^2 = \int_{\mathbb{R}^2} |\xi|^{2s} \hat{f}(\xi) \bar{\hat{f}}(\xi) \rmd \xi \leq \|f\|_{\dot{H}^{2s}} \|f\|_{L^2}$ for all $s$.  Iterating, we see that $\|f\|_{\dot{H}^s}^2 \leq  (\|f\|_{\dot{H}^{4s}} \|f\|_{L^2})^{1/2} \|f\|_{L^2}$, etc so that
\be
\|f\|_{\dot{H}^s}\leq \|f\|_{\dot{H}^{2^ks}}^\frac{1}{2^k} \|f\|_{L^2}^{1-\frac{1}{2^k}}.
\ee
Applying this to $ \|u_1-u_2\|_{\dot{H}^{1/p}}$, we see that 
\be
\|u_1-u_2\|_{\dot{H}^{1/p}}^2 \leq \|u_1-u_2\|_{\dot{H}^{2^k/p}}^\frac{2}{2^k} \|u_1-u_2\|_{L^2}^{2-\frac{2}{2^{k}}}.
\ee
Now assume that  $\|u_1-u_2\|_{\dot{H}^{s}}\leq C_*^{s+1} s^{s\eta}$ for all $s\geq 0$ and some $\eta\geq 0$.  This assumption will result in the difference being either analytic or in some Gevrey space (depending on whether $\eta=1$ or $\eta>1$ respectively). We get
\begin{align}\nonumber
\|u_1-u_2\|_{\dot{H}^{1/p}}^2 &\leq  \left(C_*^{\frac{2^k}{p} + 1} \right)^\frac{2}{2^k} \left(\left(\frac{2^k}{p}\right)^{\frac{2^k}{p} \eta} \right)^\frac{2}{2^k} \|u_1-u_2\|_{L^2}^{2-\frac{2}{2^{k}}}\leq C_*^{\frac{2}{p} +1}\left(\frac{2^k}{p}\right)^{\frac{2\eta}{p}}  \|u_1-u_2\|_{L^2}^{2}  \frac{1}{\|u_1-u_2\|_{L^2}^{\frac{2}{2^{k}}}}.
\end{align}
If $\| u_1 - u_2 \|_{L^2} > 1$, by choosing $k=1$, we have
\be
\|u_1-u_2\|_{\dot{H}^{1/p}} \leq K_{**} \|u_1 - u_2 \|_{L^2}
\ee
for some constant $K_{**}$. For $\| u_1 - u_2 \|\le 1$, choose $k$ so that 
\be
\log \frac{1}{\|u_1-u_2\|_{L^2}}  \leq 2^k \leq 4 \log \frac{1}{\|u_1-u_2\|_{L^2}}.
\ee
Then
\begin{align}\nonumber
\|u_1-u_2\|_{\dot{H}^{1/p}}^2 
&\leq C_*^{\frac{2}{p} +1}\left(\frac{4}{p}\right)^{\frac{2\eta}{p}} e^2 \|u_1-u_2\|_{L^2}^{2} \left [\log\left(  \frac{1}{\|u_1-u_2\|_{L^2}}\right) \right]^{\frac{2\eta}{p}}.
\end{align}
In conclusion 
\begin{align}\nonumber
\frac{1}{2}\frac{\rmd}{\rmd t}\|u_1-u_2\|_{L^2}^2 &\leq K_*  \|\nabla u_1\|_{L^p}  \|u_1-u_2\|_{L^2}^{2} \left ( 1 +  \left ( \left[ \log  \left(  \frac{1}{\|u_1-u_2\|_{L^2}}\right) \right ]^{\frac{2\eta}{p}}, 0 \right )_+ \right ),
\end{align}
where $(f)_+$ denotes the positive part of $f$. Now we see that if $\eta\leq \frac{p}{2}$, we have
\be
\frac{\rmd}{\rmd t} \| u_1 - u_2 \|_{L^2} \leq K_{*} \| \nabla u_1 \|_{L^p} \| u_1 - u_2 \|_{L^2} \left ( 1 + \left (\log \frac{1}{\| u_1 - u_2 \|_{L^2} } \right )_+ \right ),
\ee
and applying Osgood's lemma, we end up with the claimed bound.  Note that $\eta=1$ corresponds to analytic perturbations and $\eta>1$ corresponds to Gevrey perturbations.  
\end{proof}

Next, we show that then any limit of \emph{regularized solutions} will be smooth outside of a growing ball. In particular, possible non-uniqueness can only come from inside that ball. 
\begin{theorem}\label{ContinuationTheorem}
Let $\omega_0\in ( L^1\cap L^p)(\mathbb{R}^2)\cap C^k(\mathbb{R}^2\setminus\{0\})$ for $p>2$ and $k\geq 1$.  
Consider a family $\omega^\varepsilon\in C^\infty$ of 2D Euler solutions from regularized data $\omega_0^\varepsilon\in C^\infty$ with $C^k$ norms bounded away from $B_\delta(0)$ uniformly in $\varepsilon$ for any $\delta>0$. Let $\omega_*$ be any subsequential limit in $L^2$ of $\omega^\varepsilon$ as $\varepsilon\rightarrow 0$. Then, $\omega_*$ is a weak solution to the 2D Euler equation that is $C^k$  outside of $B_{\delta+ C t}(0)$ for some $C>0$.  
\end{theorem}

\begin{proof}

For simplicity, set $\delta=1$ and treat the case of $ C^k$ regularity.  It is possible to propagate other types of regularity such as Sobolev $H^k$, Gevrey $\mathcal{G}^{{1}/{s}}$ and analytic.
By virtue of the transport of vorticity, we have the time and $\varepsilon$ independent bound so that
 \be\label{globalL2}
\left\|\omega^\varepsilon(t)\right\|_{L^{p}}\leq \| \omega_0^\varepsilon\|_{L^p}\lesssim \| \omega_0\|_{L^p}, \qquad \text{for all}\quad \varepsilon> 0,\  t\in \mathbb{R}.
\ee Since $\omega\in L^1\cap L^{2+}$, standard elliptic estimates on the Biot-Savart law give a uniform bound on $u$:
\be\label{velbd}\|u^\varepsilon\|_{L^\infty}\leq C(\omega_0)\qquad \text{ for all} \quad \varepsilon> 0, \ t\in \mathbb{R}.
\ee Since  ${\omega}$ is transported by the bounded velocity field \eqref{velbd}, this implies 
\be\label{localLinftybound}
\|{\omega^\varepsilon}(t)\|_{L^\infty(A)}\leq \|{\omega_0^\varepsilon}\|_{L^\infty(B_{Ct}(A))} \qquad \text{for all}\quad t\geq 0 
\ee
for any open set $A\subset \mathbb{R}^2$ where $B_{\delta}(A)$ is the set of points within distance $\delta$ of $A$.  Applying this with $A = \mathbb{R}^2\setminus B_{1+C t  }(0)$, we have $B_{Ct}(A) \subset \mathbb{R}^2\setminus B_{1}(0)$ for all $t\in \mathbb{R}$.
The global $L^2$ bound on $\omega$ from \eqref{globalL2} along with the local $L^\infty$ bounds  \eqref{localLinftybound} gives local propagation of regularity.  More precisely, for $x\in  \mathbb{R}^2\setminus B_{1+Ct}(0)$, we have
\be
u (x,t) = \int_{ \mathbb{R}^2} K(x-y) \omega(y,t) (1-\chi(x-y)) \rmd y+ \int_{\mathbb{R}^2} K(y) \omega(x-y,t) \chi(y) \rmd y =: u_1 + u_2,
\ee
where $\chi(y)$ is a smooth radially decreasing function which is $1$ for $y\leq 1$ and $0$ for $y>2$.
Note 
\begin{align}
\| u_1\|_{C^{k+1}( \mathbb{R}^2\setminus B_{1+Ct}(0))}  &\leq \| \partial^{k+1} (\chi K)(\cdot) \|_{L^\infty([1,\infty))}\|\omega\|_{L^1(\mathbb{R}^2)}, \\
\| u_2\|_{C^{k+1}( \mathbb{R}^2\setminus B_{1+Ct}(0))} &\lesssim \|\omega\|_{C^{k}( \mathbb{R}^2\setminus B_{1+Ct}(0))}\|\chi K\|_{L^1(\mathbb{R}^2)}.
\end{align}
It follows that $u\in C^{k+1}( \mathbb{R}^2\setminus B_{1+Ct}(0))$ with norms bounded uniformly in $\varepsilon>0$.  Thus, we may bootstrap the initial $C^k$ regularity of the vorticity field to say that $\omega\in C^k( \mathbb{R}^2\setminus B_{1+Ct}(0))$.
That $\omega_*$ is a weak solution follows from a standard compactness argument using the bound \eqref{globalL2}. 
\end{proof}

Finally, we finish with a result which quantifies the effect that non-uniqueness emanating from a localized singularity can have in the exterior.  In the following, let $\chi_r$ be a smooth, nonnegative, radially decreasing $[0,1]$-valued function such that $\chi_r (x)$ is  $1$ for $ |x| \le r$ and $0$ for $ |x| \ge r+1$.   

\begin{theorem}  
Suppose that   $\{ \omega_0 ^\varepsilon \}_{\varepsilon>0} \subset C^\infty (\mathbb{R}^2)$  satisfying
 \begin{enumerate}
\item
$ \omega_0 ^\varepsilon \rightarrow \omega_0$ in $L^p$ for some $p>2$ as $\varepsilon \to 0$,
\item  $\{  (\omega_0 ^{\varepsilon} (1-\chi_1 ) ) \}_{\varepsilon>0 }$ is uniformly bounded in $L^1\cap C^1$. 
\end{enumerate} 
Let
 $\varepsilon_1,\varepsilon_2>0$, $\omega^{\varepsilon}$ be the resulting Euler solution and 
\be
\mathcal{E}_q (\varepsilon_1, \varepsilon_2, t) = C  \|\omega^{\varepsilon_1} (t,\cdot)-\omega^{\varepsilon_2} (t,\cdot)  \|_{L^q(B_{(2C_* +2)t + 3)}(0))}, 
 \ee
 where represent $C_*$ is the uniform $L^\infty$ bound on $\{ u ^\varepsilon(t) \}_{\varepsilon>0} $ for all time, depending only on $\sup_\varepsilon \| \omega_0 ^\varepsilon \|_{L^p}$, $p>2$. 
 Then for each $t>0$,  $\varepsilon_1,\varepsilon_2>0$ and 
 $x \in \mathbb{R}^2 \setminus B_{(2 C_* + 2) t + 3 }(0)$, we have
 \be\nonumber
 | \omega^{\varepsilon_1} (x, t) - \omega^{\varepsilon_2} (x,t) |  \le C \sum_{j=1} ^2 \| (\omega^{\varepsilon_j}_0  - \omega _0) (1-\chi_1) \|_{L^\infty} + e^{t C_*''} \left ( t \sum_{j=1} ^2 \| \omega^{\varepsilon_j}_0 - \omega_0 \|_{L^p} + \int_0 ^t \mathcal{E}_p (\varepsilon_1, \varepsilon_2, s) \rmd s \right ).
 \ee
For $x \in \mathbb{R}^2 \setminus B_{3(C_* + 1) t + 4 } (0)$ (slightly further away), we have
\begin{align}\nonumber
| \omega^{\varepsilon_1} (x, t) - \omega^{\varepsilon_2} (x,t) | &\le C \sum_{j=1} ^2 \| (\omega^{\varepsilon_j}_0  - \omega _0) (1-\chi_1) \|_{L^\infty} \\ 
 &\qquad +e^{t C_*''}  \left ( t \sum_{j=1} ^2 \| \omega^{\varepsilon_j}_0 - \omega_0 \|_{L^p} + \frac{1}{|x| - (3C_* + 2) t - 3 } \int_0 ^t \mathcal{E}_1 (\varepsilon_1, \varepsilon_2, s) \rmd s \right ). \nonumber
\end{align}
\end{theorem}

\begin{remark}
The above theorem implies that $L^1_t L^p_x$ convergence of $\omega^\epsilon (s) $ in region $B_{(2C_* + 2) s + 3} (0) $ determines uniform convergence outside $B_{(2C_* + 2) t + 3} (0)$. Also, if $x$ is sufficiently far away from $3C_*t$-ball, the effect of nonuniqueness inside the $2C_*t$-ball is small.
\end{remark}

\begin{remark}
A similar result holds provided that $\omega_0^\varepsilon$ has a uniform modulus of continuity in $B_1(0)^c$.
\end{remark}
\begin{proof}
We show estimate for difference of flows: we start from that if $x \in \mathbb{R}^2 \setminus B_{(2 C_* + 2)t + 3 } (0)$, then for $s \in [0, t]$ and any $\varepsilon >0$, 
\be
X_s ^\varepsilon (x) \in \mathbb{R}^2  \setminus B_{(C_* + 1) s + 3/2} (0).
\ee
In particular, $X_s ^{\varepsilon_1} (x)$ and $X_s ^{\varepsilon_2} (x)$ belong to the region where $u^{\varepsilon_2}(s)$ is Lipschitz. Now we write
\begin{align*}
\frac{\rmd}{\rmd t} ( X_t ^{\varepsilon_1} (x) - X_t ^{\varepsilon_2} (x) ) &= u^{\varepsilon_1} (X_t ^{\varepsilon_1} (x), t) - u^{\varepsilon_2} (X_t ^{\varepsilon_2} (x), t) \\
&= (u^{\varepsilon_1} - u^{\varepsilon_2} ) (X_t ^{\varepsilon_1} (x), t) + (u^{\varepsilon_2} (X_t ^{\varepsilon_1} (x), t) -u^{\varepsilon_2} (X_t ^{\varepsilon_2} (x), t) ), \\
|X_t ^{\varepsilon_1} (x) - X_t ^{\varepsilon_2} (x) | &\le \int_0 ^t  | (u^{\varepsilon_1} - u^{\varepsilon_2}) ( X_s^{\varepsilon_1} (x),s) | ds + \int_0 ^t C_* ' ( |X_s ^{\varepsilon_1 } (x) - X_s ^{\varepsilon_2} (x) | ) \rmd s 
\end{align*}
Denoting $A_t = X^{-1}_t$ and invoking Biot-Savart's law, we have
\be
\begin{split}
(u^{\varepsilon_1} - u^{\varepsilon_2}) ( X_s^{\varepsilon_1} (x),s) &= \int_{\mathbb{R}^2} \frac{(X_s ^{\varepsilon_1} (x) - y)^{\perp} } {|X_s ^{\varepsilon_1} (x) - y|^2} (\omega_0 ^{\varepsilon_1} - \omega_0 ^{\varepsilon_2}) ( A_s ^{\varepsilon_1} (y))\rmd y\\
& + \int_{\mathbb{R}^2} \frac{(X_s ^{\varepsilon_1} (x) - y)^\perp}{|X_s ^{\varepsilon_1} (x) - y|^2} \big(\omega_0^{\varepsilon_2} (A_s ^{\varepsilon_1} (y) ) - \omega_0^{\varepsilon_2} (A_s ^{\varepsilon_2} (y) ) \big) dy =: I_1 + I_2.
\end{split}
\ee
We notice that for $p>2$,
\be
|I_1| \le C \| \omega_0 ^{\varepsilon_1} - \omega_0^{\varepsilon_2} \|_{L^p}.
\ee
Since $\omega_0 ^{\varepsilon_2} \rightarrow \omega_0 $ in $L^p$, $p>2$, $I_2$ is controlled by
\be
I_2\leq C \| \omega_0 ^{\varepsilon_2} - \omega_0 \|_{L^p} + I_2 ',
\ee
where $I_2 '$ is
\be
\begin{split}
I_2 ' &=\left( \int_{|y| \le {(2C_* + 2) s + 3}  } + \int_{|y| > {(2C_* + 2)s + 3 }  }\right)\frac{(X_s ^{\varepsilon_1} (x) - y)^\perp}{|X_s ^{\varepsilon_1} (x) - y|^2} \big(\omega_0 (A_s ^{\varepsilon_1} (y) ) - \omega_0 (A_s ^{\varepsilon_2} (y) ) \big)\rmd y.
\end{split}
\ee
The second integral above can be controlled by
\be\nonumber
C \sup_{y \in \mathbb{R}^2 \setminus B_{(2C_* + 2) s + 3} (0) } |X_s ^{\varepsilon_1} (y) - X_s ^{\varepsilon_2} (y) | \| \nabla (\omega_0  (1- \chi_1 ) )\|_{L^p} =  \sup_{y \in \mathbb{R}^2 \setminus B_{(2C_* + 2) s + 3} (0) } |X_s ^{\varepsilon_1} (y) - X_s ^{\varepsilon_2} (y) | C_{***} '
\ee
for $p>2$, using the characterization of $W^{1,p}$ space ($|f(x) - f(y) | \le (M\nabla f (x) + M \nabla f(y) ) |x-y|$; one can refer to \cite{H1996}, \cite{CD2008} and references therein for example) and that $X_s ^{\varepsilon_2} (y) \in \mathbb{R}^2 \setminus B_1 (0)$ for $|y| > (2C_* + 2) s + 3$. The first integral is controlled by $ \mathcal{E}_p (\varepsilon_1, \varepsilon_2, s)$.
If $|x| > 3(C_*+1) t + 4$, we have a better control: the first integral is controlled by
\be
\frac{1}{|x| - (3C_* + 2) s - 3} \mathcal{E}_1 (\varepsilon_1, \varepsilon_2, s).
\ee
Therefore, if we let 
\be
\rho(\varepsilon_1, \varepsilon_2, t) = \sup_{x \in \mathbb{R}^2 \setminus B_{(2C_* + 2)t + 3} (0) } | X_t ^{\varepsilon_1} (x) - X_t ^{\varepsilon_2} (x) |,
\ee
we have
\begin{align*}
\rho(\varepsilon_1, \varepsilon_2, t) &\le Ct (\| \omega_0 ^{\varepsilon_1} - \omega_0 \|_{L^p} + \| \omega_0 ^{\varepsilon_2} - \omega_0 \|_{L^p}) + \int_0 ^t \mathcal{E}_p (\varepsilon_1, \varepsilon_2, s) \rmd s + \int_0 ^t (C_{***} ' + C_* ') \rho(\varepsilon_1, \varepsilon_2, s)  \rmd s.
\end{align*}
Then, by Gronwall's lemma, we have
\be\label{flowestimate1}
\rho(\varepsilon_1, \varepsilon_2, t) \le e^{(C_*' + C_{***} ' )t } \left ( Ct \sum_{j=1}^2 \| \omega_0^{\epsilon_j } - \omega_0 \|_{L^p } + \int_0 ^t \mathcal{E}_p (\epsilon_1, \epsilon_2, s) \rmd s \right ).
\ee
Moreover, if $|x| > 3(C_* + 1) t + 4$, we have
\be\label{flowestimate2}
|X_t ^{\varepsilon_1} (x) - X_t ^{\varepsilon_2} (x) | \le e^{(C_*' + C_{***} ' )t } \left (  C t \sum_{j=1} ^2\| \omega_0^{\varepsilon_{j}}  - \omega_0\|_{L^p} + \frac{1}{|x| - (3C_* + 2) t - 3 } \int_0 ^t \mathcal{E}_1 (\varepsilon_1, \varepsilon_2, s) \rmd s \right ).
\ee
In particular, for any $\varepsilon> 0$ and $x \in \mathbb{R}^2 \setminus B_{(2C_* + 2) t + 3 } (0)$, we see that $\omega_0 ^{\varepsilon} \chi_1 \circ X_t ^{\varepsilon} (x) = 0$. Thus, if
\be
\sup_\varepsilon \| \nabla (\omega^\varepsilon_0 (1-\chi_1 ) ) \|_{L^\infty} \le C''_* < \infty,
\ee
then
\begin{align*}
\omega^{\varepsilon_1} (x, t) - \omega^{\varepsilon_2} (x, t) &= \omega_0 ^{\varepsilon_1} (1-\chi_1 ) \circ A_t ^{\varepsilon_1} (x)- \omega_0 ^{\varepsilon_2} (1-\chi_1 ) \circ A_t ^{\varepsilon_2} (x) \\
&= (\omega_0 ^{\varepsilon_1} - \omega_0 ^{\varepsilon_2} ) (1-\chi_1 ) \circ A_t ^{\varepsilon_1} (x) \\
&\qquad + ( \omega_0 ^{\varepsilon_2} (1-\chi_1 ) \circ A_t ^{\varepsilon_1} (x) - \omega_0 ^{\varepsilon_2} (1-\chi_1 ) \circ A_t ^{\varepsilon_2} (x) ), \\
| \omega^{\varepsilon_1} (x, t) - \omega^{\varepsilon_2} (x,t) | &\le \|  (\omega^{\varepsilon_1} - \omega^{\varepsilon_2}) (1- \chi_1 ) \|_{L^\infty} + |X_t ^{\varepsilon_1} (x) - X_t ^{\varepsilon_2} (x) | C_* ''.
\end{align*}
Applying the above estimate for $|X_t ^{\varepsilon_1} (x) - X_t ^{\varepsilon_2} (x)|$ gives the desired conclusion.
\end{proof}

\section*{Acknowledgements}

The research of T.D. was partially supported by NSF-DMS grant 2106233 and the Charles Simonyi Endowment at the Institute for Advanced Study.
 The research of T.E. was partially supported by the NSF grants DMS-2043024 and DMS-2124748  as well as the Alfred P. Sloan Foundation.

\bibliographystyle{amsplain}

\begin{thebibliography}{10} 
\bibitem{notesVishik} 
Albritton, D., Brué, E., Colombo, M., De Lellis, C., Giri, V., Janisch, M., \& Kwon, H. (2021). Instability and nonuniqueness for the $2 d $ Euler equations in vorticity form, after M. Vishik. arXiv preprint arXiv:2112.04943.
\bibitem{AB08} 
L.Ambrosio and P.Bernard, Uniqueness of signed measures solving the continuity equation for Osgood vector fields, Atti Accad. Naz. Lincei Rend. Lincei Mat. Appl. 19 (2008), no. 3, 237–245.
\bibitem{BCD11} 
H.Bahouri, J.Y.Chemin and  R.Danchin. (2011). Fourier analysis and nonlinear partial differential equations (Vol. 343). Springer Science \& Business Media.
\bibitem{BC} 
A.L.Bertozzi and P.Constantin. Global regularity for vortex patches. Comm. Math. Phys.,
152(1):19–28, 1993.
\bibitem{BM} 
A.Bressan and R.Murray. On self-similar solutions to the incompressible Euler equations. Journal of Differential Equations 269.6 (2020): 5142-5203.
\bibitem{CC21} 
 L.Caravenna and G.Crippa, A directional Lipschitz extension lemma, with applications to uniqueness and Lagrangianity for the continuity equation, Comm. Partial Differential Equations 46 (2021), no. 8, 1488–1520.
\bibitem{CJ21} 
D.Chae and I.--J.Jeong. Preservation of log-H\"{o}lder coefficients of the vorticity in the transport equation. arXiv preprint arXiv:2104.11484 (2021).
\bibitem{Chem} 
Jean-Yves Chemin. Persistance de structures g\'{e}om\'{e}triques dans les fluides incompressibles bidimensionnels. Ann. Sci. Ecole Norm. Sup. (4) , 26(4):517–542, 1993.
\bibitem{CCS21} 
G.Ciampa, G.Crippa, and S.Spirito. Strong convergence of the vorticity for the 2D Euler Equations in the inviscid limit. Archive for Rational Mechanics and Analysis 240.1 (2021): 295-326.
\bibitem{CDE2019}
P.Constantin, T.D.Drivas, and T.M.Elgindi. Inviscid Limit of Vorticity Distributions in the Yudovich Class. Communications on Pure and Applied Mathematics (2020).
\bibitem{CS21} 
G.Crippa and  G.Stefani. An elementary proof of existence and uniqueness for the Euler flow in localized Yudovich spaces. arXiv preprint arXiv:2110.15648 (2021).
\bibitem{DM87} 
 R.DiPerna and A.Majda. 1987 Concentrations in regularizations for 2-D incompressible flow Commun. Pure
Appl. Math. XL 301–45
\bibitem{D15} 
S.Denisov. Double exponential growth of the vorticity gradient for the two-dimensional Euler equation. Proceedings of the American Mathematical Society 143.3 (2015): 1199-1210.
\bibitem{EJ} 
T.M.Elgindi and I-J Jeong. On singular vortex patches, I: Well-posedness issues. arXiv preprint arXiv:1903.00833 (2019).
\bibitem{Elling} 
V.Elling. Self-similar 2d Euler solutions with mixed-sign vorticity. Communications in Mathematical Physics 348.1 (2016): 27-68.
\bibitem{H02} 
P.Hartman. Ordinary differential equations. Society for Industrial and Applied Mathematics, 2002.
\bibitem{MP91} 
C.Marchioro and M.Pulvirenti. On the vortex–wave system. Mechanics, analysis and geometry: 200 years after Lagrange. Elsevier, 1991. 79-95.
\bibitem{sch} 
S.Schochet. The point‐vortex method for periodic weak solutions of the 2‐D Euler equations. Communications on pure and applied mathematics 49.9 (1996): 911-965.
\bibitem{S94} 
Philippe Serfati. Une preuve directe d’existence globale des vortex patches 2D. C. R. Acad. Sci.
Paris S\'{e}r. I Math., 318(6):515–518, 1994.
\bibitem{S97} 
A.Shnirelman, Evolution of singularities, generalized Liapunov function and generalized integral for an ideal incompressible fluid. American Journal of Mathematics, 119(3), 579-608,  (1997).
\bibitem{Vishik1} 
M.Vishik. Instability and non-uniqueness in the Cauchy problem for the Euler equations of an ideal incompressible fluid. Part I. 2018. arXiv: 1805.09426
\bibitem{Vishik2} 
M.Vishik. Instability and non-uniqueness in the Cauchy problem for the Euler equations of an ideal incompressible fluid. Part II. 2018. arXiv: 1805.09440
\bibitem{Y95} 
V.I.Yudovich, Uniqueness theorem for the basic nonstationary problem in the dynamics of an ideal incompressible fluid. Mathematical Research Letters 2.1 (1995): 27-38.
\bibitem{H1996}
P. Hajlasz, Sobolev spaces on an arbitrary metric space. Potential Analysis 5:403-415, 1996.
\bibitem{CD2008} G. Crippa and C. De Lellis, Estimates and regularity results for the Diperna-Lions flow. J. reine angew. Math. 616(2008), 15-46.
\end{thebibliography}

\end{document}